\documentclass[12pt, psamsfonts]{amsart}

\usepackage{amsmath,amsthm}
\usepackage{amssymb}

\usepackage{enumitem}

\usepackage{xcolor}

\usepackage{ulem}

\usepackage[T2A]{fontenc}
\usepackage[utf8]{inputenc}

\newtheorem{theorem}{Theorem}[section]
\newtheorem{corollary}[theorem]{Corollary}
\newtheorem{lemma}[theorem]{Lemma}

\newtheorem{proposition}[theorem]{Proposition}

\theoremstyle{definition}

\newtheorem{remark}[theorem]{Remark}
\newtheorem{example}[theorem]{Example}

\numberwithin{equation}{section}

\newcommand\eq[2]{{\begin{equation}\label{eq:
#1}{#2}\end{equation}}}

\newcommand {\equ}[1] {\eqref{eq: #1}}

\newcommand{\BB}{{\mathcal{B}}}

\newcommand {\new}[1] {\textcolor{blue}{#1}}
\newcommand {\comm}[1] {\textcolor{red}{#1}}

\newcommand{\norm}[1]{{\|{#1}\|}}
\newcommand{\dz}[1]{{\langle{#1}\rangle}}

\newcommand{\Q}{{\mathbb {Q}}}

\newcommand{\R}{{\mathbb{R}}}

\newcommand{\Z}{{\mathbb{Z}}}

\newcommand{\N}{{\mathbb{N}}}

\newcommand{\SL}{\operatorname{SL}}

\newcommand{\spa}{{\rm span}}

\newcommand{\df}{{\, \stackrel{\mathrm{def}}{=}\, }}

\newcommand{\nz}{\smallsetminus\{0\}}

\newtheorem{innercustomthm}{Theorem}
\newenvironment{customthm}[1]
  {\renewcommand\theinnercustomthm{#1}\innercustomthm}
  {\endinnercustomthm}

\newcommand\hd{Hausdorff dimension}

\newcommand\ba{badly approximable}

\newcommand\da{Diophantine approximation}
\newcommand\de{Diophantine exponent}

\newcommand\di{Diophantine}

\newcommand\hs{homogeneous space}

\newcommand {\ignore}[1] {}

\newtheorem*{remark*}{Remark}

\begin{document}


\baselineskip=17pt


\title{Singular vectors on manifolds and fractals}

\author{Dmitry Kleinbock}
\address{Department of Mathematics, Brandeis University, 
Waltham MA, USA 02454-9110} 
\email{kleinboc@brandeis.edu}
\author{Nikolay Moshchevitin}
\address{Moscow State University, 
Leninskie Gory 1, Moscow, Russia,   1119991, 
and \linebreak Astrakhan State University,
Tatishcheva 20a,
Astrakhan, Russia, 414056} 
\email{moshchevitin@gmail.com}
\author{Barak Weiss}
\address{Department of Mathematics, Tel Aviv University, Tel Aviv, Israel} 
\email{barakw@post.tau.ac.il}

\date{December 2019}

\maketitle


\renewcommand{\thefootnote}{}



\renewcommand{\thefootnote}{\arabic{footnote}}
\setcounter{footnote}{0}


\begin{abstract}
 We generalize Khintchine's method of constructing totally irrational
 singular vectors and linear forms.  The main result of the paper
 shows existence of totally irrational vectors and linear forms with
 large uniform \di\ exponents on certain subsets of $\R^n$, in
   particular on any analytic submanifold of $\R^n$ of dimension
   $\ge 2$  which is not contained in a proper rational affine
 subspace.
 \end{abstract}

\section{Introduction}\label{intro}
\subsection{Singular vectors and uniform \di\ exponents}
 In this paper we consider uniform rational approximations  to
 $n$-tuples of real numbers. 
Denote by $\dz{x}$    the distance  from $x\in\R$ to the nearest integer, and for
 $\pmb{x} = (x_1,\dots,x_n)\in\R^n$ and  $\pmb{y} = (y_1,\dots,y_n)\in\R^n$  let 
 $$\dz{\pmb{x}}\df \big(\dz{x_1},\dots,\dz{x_n}\big), \ \ \ \
 \norm{\pmb{x}}\df \max_{1\le j \le n}|x_j|, \ \ \ \
 \pmb{x}\cdot\pmb{y}\df x_1y_1+\cdots x_ny_n.$$ 
 
A vector $\pmb{\xi} =  (\xi_1,\dots ,\xi_n)$ is called {\sl singular}
 if  for every 
$c>0$  the system of inequalities
\eq{singvector}{
\norm{\dz{q\pmb{\xi}}} \le c t^{-1/n},\quad 0 < q \le t
}has an integer solution $q$
for any sufficiently large
$t$. Equivalently (in view of  Khintchine's Transference Principle
\cite{H1, C}), $\pmb{\xi}$ is   singular  if for every  
$c > 0$  the system of inequalities
\eq{singform}{
\dz{\pmb{q}\cdot\pmb{\xi}}
\le c t^{-n},\quad 0<\norm{\pmb{q}} \le t}has 
a solution $\pmb{q} \in\Z^n$
for any sufficiently large
$t$. We note  that from Dirichlet's theorem, or, alternatively, from
Minkowski's convex body theorem, it follows that when $c = 1$, for all
$t\ge 1$ both \equ{singvector} and \equ{singform} have integer
solutions. It is well-known that the set of singular real numbers
coincides with $\Q$; thus in what follows we will assume that $n \geq
2$. 

 It was observed by Khintchine, see \cite[Ch.\ V, \S7]{C}
 that 
{singular vectors} form a set of Lebesgue measure zero.  
 \ignore{Let
 \begin{equation*}
\psi_{\pmb{\xi}} (t) \df
\min_{{q} \in \mathbb{Z}_+:\,\,
q\le t}
\,\,\,\,\,
\max_{1\le j \le n}
||\xi_jq ||
\end{equation*}
be the {\sl irrationality measure function related to the simultaneous approximation} to $\xi_1,\dots ,\xi_n$
(where
$||\cdot ||$ stands for the distance to the nearest integer).
Also let
\begin{equation}\label{f1}
\psi_{\pmb{\xi}}^* (t) \df
\min_{\pmb{q} =(q_1,\dots ,q_{n}) \in \mathbb{Z}^n:\,\,
0<\max_{1\le j \le n}|q_j|\le t}
\,\,\,\,\,
||\xi_1q_1+\dots +\xi_n q_n ||
\end{equation}
be the {\sl irrationality measure function related to the dual (linear form) approximation} to $\xi_1,\dots ,\xi_n$.
From Dirichlet's theorem, or, alternatively, from Minkowski's convex body theorem, it follows that for all $t\ge 1 $ one has
 \begin{equation}\label{f3}
\psi_{\pmb{\xi}} (t)\le \frac{1}{t^{{1}/{n}}}\,\,\,\,\,\,
\text{and}\,\,\,\,\,
\psi_{\pmb{\xi}}^* (t)\le \frac{1}{t^n}.
\end{equation}
$\pmb{\xi}$ is called {\sl singular} if \eq{sing}{\lim_{t\to\infty}t^{{1}/{n}}\psi_{\pmb{\xi}} (t) = 0 \quad\Longleftrightarrow\quad \lim_{t\to\infty}t^{n}\psi_{\pmb{\xi}}^* (t) = 0,}
the two statements above being equivalent in view of Khintchine's
transference principle \comm{[reference]}. 
 It was observed by Khintchine \comm{[reference]} that those
 $\pmb{\xi}$ form a set of Lebesgue measure zero.}   
 One reason why singular vectors are an interesting object of study is
 their connection with homogeneous dynamics. It was showed by Dani
 \cite{d} that  
$\pmb{\xi}$ is   singular if and only if the trajectory of a certain lattice in $\R^{n+1}$ constructed from $\pmb{\xi}$ diverges
 (i.e.\ leaves every compact subset of the space of lattices). We will
 not exploit this connection in the present paper; see however
 \cite{KW, W, DFSU} for generalizations and further discussions. In
   particular, the \hd\ of  the set of singular vectors in $\R^n$
     was only relatively recently   shown by Cheung and Chevallier
     \cite{CC} to be equal to $\frac{n^2}{n+1}$; see also an earlier
   work of Cheung \cite{C} settling the case $n=2$.

One can also introduce different `levels of singularity' of vectors
  $\pmb{\xi}\in\R^n$  by considering {exponents of 
uniform
Diophantine approximation. 
Namely, one defines $ \hat{\omega}(\pmb{\xi})$,  the 
{\sl uniform  exponent  of $\pmb{\xi}$ in the sense of simultaneous
  approximation}, as  
the supremum of $\gamma> 0$ for which the system of inequalities    
  $$
\norm{\dz{q\pmb{\xi}}} \le t^{-\gamma},\quad 0 < q \le t 
$$ 
 has an integer solution $ q$ for all $t$ large enough. Likewise,  
  $ \hat{\omega}^*(\pmb{\xi})$, the 
{\sl uniform  exponent  of $\pmb{\xi}$ in the sense of dual
  approximation}, is defined as the supremum of such  $\gamma$ for
which the system of inequalities 
$$
\dz{\pmb{q}\cdot\pmb{\xi}}
\le  t^{-\gamma},\quad 0<\norm{\pmb{q}} \le t 
$$ 
has an integer solution $ \pmb{q}
$ for all $t$ large enough.
  It is clear that always
\begin{equation}\label{u}
 \hat{\omega}(\pmb{\xi})\ge 1/n    \,\,\,\,\,\,\text{and}\,\,\,\,\,\,   
 \hat{\omega}^*(\pmb{\xi})\ge n.
\end{equation}
{In  \cite{DFSU} vectors $\pmb{\xi}$ satisfying $ \hat{\omega}(\pmb{\xi}) > 1/n$ (equivalently,  $\hat{\omega}^*(\pmb{\xi}) > n$) were called {\sl very singular}; clearly very singular implies singular.} See \cite[Theorem 1.3]{DFSU} for an interpretation of the quantities  $ \hat{\omega}(\pmb{\xi})$ and  $ \hat{\omega}^*(\pmb{\xi})$ in terms of the rate of divergence of certain trajectories in the space of lattices.}

\subsection{Theorems of Khintchine and Jarn\'{\i}k } 
Let us say that $\pmb{\xi}\in\R^n$  is {\sl totally irrational} if $1,\xi_1,\dots ,\xi_n$ are  linearly independent  over $\mathbb{Q}$. 
It is easy to see that for not totally irrational vectors $\pmb{\xi}$ one has 
{
$$
\hat{\omega}^*(\pmb{\xi}) = \infty \quad \text{ and }\quad
  \hat{\omega}(\pmb{\xi}) \ge \frac1{n-1};$$ 
  in particular, they clearly are very singular.} On the other hand, 
in \cite{tb}  Jarn\'{\i}k observed that  for   
totally irrational $\pmb{\xi}$ 
one has the upper bound
\eq{upperbound}{
 \hat{\omega}(\pmb{\xi}) \le 1.
}
In a fundamental paper \cite{H1} in the case $ n = 2$ 
 Khintchine  discovered the phenomenon of existence 
of  {very singular} totally irrational vectors.
This was later generalized  by Jarn\'{\i}k  to the setting of systems
of linear forms \cite{J}. The following two theorems constitute a
special case of Jarn\'{\i}k's result. 

\begin{customthm}{A}\label{A}  \it
 There exist  {continuum many 
 totally irrational $\pmb{\xi}  \in \mathbb{R}^n$ such that $\hat{\omega}^*(\pmb{\xi}) = \infty$}.  
\end{customthm}
 
 \begin{customthm}{B}\label{B}  \it
There exist  {continuum many 
 totally irrational $\pmb{\xi}  \in \mathbb{R}^n$ such that $\hat{\omega}(\pmb{\xi}) = 1$}.\end{customthm}

Here we should note that in the case $n=2$  Khintchine 
deduced Theorem \ref{B} from Theorem \ref{A} by means of a transference
argument. However for $n>2$ Jarn\'{\i}k proved Theorem \ref{B}
directly, without using transference. In fact, the transference
argument from Jarn\'{\i}k's paper \cite{tb}, 
 {which can also be found in the monograph by Cassels
\cite[Ch.\ V, \S2, Thm.\ II]{C}, when} 
applied to Theorem \ref{A}  {gives a weaker conclusion  $\hat{\omega}(\pmb{\xi}) \ge \frac1{n-1}$.}

\ignore{A similar general result for arbitrary values of $n$ was
  documented by Jarn\'{\i}k in \cite{J}, where he showed that  
for any $n \ge 2$ under the conditions of Theorems  \ref{A} and \ref{B} 
there exists an uncountable dense subset
of totally irrational vectors $\pmb{\xi}$ 
satisfying respectively (\ref{odin}) or (\ref{dwa}).
(In fact, Jarn\'{\i}k dealt with an even more general situation with
$m$ linear forms in $n$ variables).} 
\smallskip

Further results, generalizations and applications are discussed in
Cassels' book \cite{C} and in a survey by the second-named author
\cite{m}.  
We note that Khintchine's method was used by Dani \cite{d} and later
by the third-named \cite{W} to exhibit   {rapidly} divergent
trajectories of diagonalizable  
semigroups on \hs s of higher rank semisimple Lie groups.

\subsection{Approximation on manifolds and fractals} \
A recurrent theme in \da\ is  the introduction of restrictions on the
vector $\pmb{\xi}$, for instance by imposing a functional dependence between its 
components, or restrictions on the digital expansion of
its coefficients. In other terms, one is interested in the Diophantine
properties of vectors $\pmb{\xi}$ which are known to lie in a certain
subset of $\R^n$, such as a fractal or a smooth submanifold.  
See \cite{BD} for history and references, and \cite{KMa, KLW, K} for 
developments utilizing dynamics on the space of lattices, and in
particular quantitative non-divergence estimates.

As far as 
singular vectors  on fractals or manifolds go, only a few results have been known
until recently. Davenport and Schmidt \cite[Theorem
3]{Davenport-Schmidt} proved that almost all vectors of the form $ 
(x,x^2)$ are not singular. This was later extended to other manifolds
\cite{Ba1, Ba2, DRV, Bu}. 
%
Recall that a smooth submanifold of
      $\R^n$ is called {\sl nondegenerate} if at its Lebesgue-almost
      every point partial derivatives of its parametrizing map up to
      some order span $\R^n$; if $M$ is connected and real analytic,
      this is equivalent to not lying in any proper affine
      subspace (we define real analytic manifolds in \S \ref{an}). Using
      quantitative non-divergence results obtained in 
      \cite{KLW}, two of the authors in \cite[Theorem 1.1]{KW}
 generalized the results of Davenport and Schmidt, proving
that the intersection of the set of singular vectors 
with any smooth nondegenerate manifold has measure zero. They also
showed that on a large class of fractal sets, the set of singular vectors has
measure zero with respect to the Hausdorff measure on the fractal. 

A natural {question}  to ask is whether  the above intersection is in
fact nontrivial, that is, not contained 
in the set of totally irrational vectors.

The only examples of curves on which nontrivial singular vectors have
been exhibited are rational quadrics in $\R^2$  such as the parabola
$\{(x,x^2): x\in\R\}$. This was done by Roy \cite{R1, R2}. 
His result for the parabola was optimal, in the sense that he
  exhibited the least upper bounds for the sets
  $\{\hat{\omega}(\pmb{\xi})\}$ and $\{\hat{\omega}^*(\pmb{\xi})\}$
  where $\pmb{\xi}$ runs through all totally irrational vectors of the
  form $(x,x^2)$.
Optimal results for quadric hypersurfaces in $\R^n$ were  very
recently obtained by  Poëls and Roy \cite{R1u,R2u}, complementing
upper estimates for uniform \de s found earlier by
two of the authors \cite{KM}, see \S \ref{upper}.

\smallskip

For a quite general class of higher-dimensional real analytic
manifolds this question was addressed 
in \cite[Theorem 1.2]{KW}: 

\begin{customthm}{C}\label{C} \it Let ${\mathcal{S}}$ be a connected real analytic
  submanifold   of $\R^n$ of 
dimension at least $2$  which is not contained in any proper rational affine
subspace of $\R^n$. Then there exists a totally irrational singular
vector $\pmb{\xi}  \in {\mathcal{S}}$. 
Moreover, one can find uncountably many  such $\pmb{\xi} $  {with
  \eq{exponentkw}{\hat{\omega}(\pmb{\xi}) \ge \frac{n^2+1}{n(n^2-1)} =
    \frac1n + \frac2{n(n^2-1)}  
.
}}
  \end{customthm}
This was actually done in the context of weighted approximation, see \S\ref{wts}.
The `moreover' part was not written explicitly in \cite{KW}, but can
  be easily derived from  \cite[Corollary 5.2 and Remark 5.4]{KW}. However the
  proof given in \cite{KW} contains a gap, and one of
  the goals of the present paper is to rectify it by providing a
  complete proof of a stronger statement. We will discuss the gap in
  the proof at the end of   \S \ref{pf}. 

\subsection{The main result} 
We now formulate a general result,  which extends
  Theorem~\ref{A} to quite general subsets ${S}\subset\R^n$, and
    from which a stronger version of Theorem  \ref{C} follows. The
    conditions on $S$ will be phrased in terms 
    of its intersections with rational affine hyperplanes.  
If $\mathbf{m} = (m_0, m_1, \ldots,
m_n) \in \Z^{n+1}$ is a primitive vector, 
we will denote 
by  $A_{\mathbf{m}}$ the hyperplane
\eq{eq: def hyperplane}{
  A_{\mathbf{m}} \df \left\{\pmb{\xi} \in \R^n : 
    \sum_{i=1}^n m_i\xi_i = m_0 \right\},}
and write 
\eq{eq: slightly}{
  \left|A_{\mathbf{m}}\right| \df 
{ \|(m_1, \ldots, m_n)\|}
  .}
We will also work with a generalized version of the uniform exponent
for dual approximation. 
Let $\Phi : {\mathbb{Z}}^n{\nz}\to \mathbb{R}_+$ be a proper function, that is
\begin{equation}\label{ccoo}
\text{ the set } \,\ 
 \{ \pmb{q} \in \mathbb{Z}^n:\,\,\, \Phi (\pmb{q} ) \le C\}
\,\,\,
\text{is finite for any}\,\,\, C >0.
\end{equation}
In accordance with $\Phi$ we define the following {\sl irrationality
  measure function} 
\begin{equation}\label{0p0}
\psi_{\Phi,\,\pmb{\xi}} (t) \df
\min_{\pmb{q} 
\in \mathbb{Z}^n\smallsetminus\{\pmb{0}
\},\,\,
\Phi(\pmb{q})\le t}
\,\,\,\,\,
\dz{\pmb{q}\cdot\pmb{\xi}}
.
\end{equation}
For example for
${ \Phi (\pmb{q}) = 
\norm{\pmb{q}}}$
the {function $\psi_{\norm{\cdot},\,\pmb{\xi}}$ can be used to define
  the uniform exponent of $\pmb{\xi}$ in the sense of dual
  approximation:  
\eq{dualexponent}{
  \hat{\omega}^*(\pmb{\xi}) = \sup\left\{ \gamma: \limsup_{t\to
    \infty}t^\gamma \psi_{\norm{\cdot},\,\pmb{\xi}}(t) 
  < \infty
  \right\}.
}

Recall that $S \subset \R^n$ is called {\sl locally closed} if there
is an open set $\mathcal{W}$ such that $S= \overline{ S} \cap
\mathcal{W}$. The following is our main result.
\begin{theorem}\label{thm: abstract}
  Let $S \subset \R^n$  be a nonempty locally closed subset,  let $
  \{L_1, L_2, \ldots\}$ and $
  \{L'_1, L'_2, \ldots\}$ be disjoint collections of distinct closed
  subsets of $S$, each of which is contained in a rational
  affine hyperplane in $\R^n$, and for each $i$ let $A_i$ be a
  rational affine hyperplane containing $L_i$. Assume the following hold:
  \begin{itemize}
  \item[\rm (a)]
    \eq{eq: Ls cover As}{
      \bigcup_{i} L_i \cup \bigcup_{j} L'_j = \{\pmb{x} \in S :
      \pmb{x} \text{ is contained in a rational affine hyperplane} \}.
      }
  \item[\rm (b)]
    For each $i$ and each $T>0$,
    $$ L_i = \overline{\bigcup_{|A_j| > T } L_i \cap L_j};$$
    \item[\rm (c)] For each $i$, and for any finite subsets of indices $F,
      F'$ with $i\notin F$, we have 
     \eq{densityinL}{L_i = \overline{ L_i
        \smallsetminus \left(\bigcup_{k \in F} L_k \cup \bigcup_{k' \in F'} L'_{k'}\right)};}
    \item[\rm (d)]
$\bigcup_i L_i$ is dense in $S$. 
  \end{itemize}
Then 
for arbitrary $\Phi :\mathbb{Z}^n\to \mathbb{R}_+$ satisfying
\eqref{ccoo} and 
for any {non-increasing}
 function $ \varphi:\mathbb{R}_+\to \mathbb{R}_+$,
 there exist uncountably many  totally irrational  $\pmb{\xi}  \in
 {S}$ such that $\psi_{{\Phi,}\pmb{\xi}} (t) \le \varphi (t)\text{ for
   all large enough $t $.}$  
\end{theorem}
An application of Theorem \ref{thm: abstract} to ${ \Phi (\pmb{q}) =
\norm{\pmb{q}}}$,
in view of \equ{dualexponent}, immediately produces
 
 \begin{corollary}\label{manifolds2}    Let $S \subset \R^n$ for which
   there exist collections $\{L_i\}, \, \{L'_j\}, \, \{A_i\}$ satisfying
   the conditions of Theorem \ref {thm: abstract}. Then  
 there exist uncountably many  totally irrational  $\pmb{\xi}  \in
 {S}$ such that
$\hat{\omega}^*(\pmb{\xi}) = \infty$.
 \end{corollary}

From this, a standard 
 transference argument from \cite{tb} and 
 \cite[Ch.\ V, \S2, Thm.\ II]{C} readily gives

\begin{corollary}\label{manifolds3}    Let $S$ be as in Corollary \ref{manifolds2}. Then 
 there exist uncountably many  totally irrational  $\pmb{\xi}  \in {S}$ such that
 $\hat{\omega}(\pmb{\xi})\ge \frac{1}{n-1}$. 
\end{corollary}

We note that 
the above corollary
gives a stronger statement than Theorem \ref{C}, since the exponent  
$\frac{n^2+1}{n(n^2-1)} 
= {\frac1{n-1} - \frac1{n(n+1)}}$ appearing in \equ{exponentkw} is
strictly smaller than $\frac1{n-1}$.

\subsection{Approximation with weights}\label{wts}  One advantage of
the general setup of Theorem \ref{thm: abstract} is the possibility to
extend our results to approximation with weights. 
The weighted setting in Diophantine approximation was initiated 
 by  Schmidt \cite{Sch} and became very popular during recent decades,
 see e.g.\ \cite{Kweights}. 
Consider 
\begin{equation}\label{tw1}
\pmb{s} = (
s_1,\dots ,s_n ) \in (0,1)^n,\,\,\,\,\, s_1+\dots +s_n = 1,
\end{equation}
and put
\begin{equation}\label{tw2}
\rho \df \max_{1\le j \le n} s_j,\,\,\,\,\,
\delta \df \min_{1\le j \le n} s_j.
\end{equation}
Then introduce the {\sl $\pmb{s}$-quasinorm} $\norm{\cdot}_{\pmb{s}}$
on $\R^n$ by $$\norm{\pmb{x}}_{\pmb{s}}\df \max_{1\le j \le
  n}|x_j|^{1/s_j}.$$ 
Clearly $\norm{\pmb{x}}_{\pmb{s}} = \norm{\pmb{x}}^n$ when $\pmb{s} =
\left(\frac1n,\dots,\frac1n \right)$.  
Now we define the 
 {\sl weighted uniform exponent
 $
  \hat{\omega}_{\pmb{s}}(\pmb{\xi})$
 for simultaneous approximation}  as
 the  supremum  of those $\gamma$ for which the system of inequalities
$$
\norm{\dz{q\pmb{\xi}}}_{\pmb{s}} \le t^{-n\gamma},\quad 0 < q \le t 
$$
has a  solution $ {q} \in \mathbb{Z}_+$ for all $t$ large enough, and
the {\sl weighted uniform    exponent} $
 \hat{\omega}^*_{\pmb{s}}(\pmb{\xi})$
of a linear form $\pmb{\xi}$ as the  supremum  of those $\gamma$ for
which the system of inequalities 
$$
{\dz{\pmb{q}\cdot\pmb{\xi}}}
\le  t^{-\gamma},\quad 0< \norm{\pmb{q}}_{\pmb{s}} \le t^n 
$$
has a  solution $ \pmb{q} \in \mathbb{Z}^n$ for all $t$ large enough. 
Analogously to \eqref{u} {and \equ{upperbound}, for totally irrational
  $\pmb{\xi}$} one always has 
\begin{equation*}\label{mo}
 \hat{\omega}_{\pmb{s}}^*(\pmb{\xi})\ge n    \,\,\,\,\,\,\text{and}\,\,\,\,\,\,   \frac{1}{n}\le \hat{\omega}_{\pmb{s}}(\pmb{\xi})\le \frac{1}{\rho n}.
\end{equation*}

Now, in order to construct vectors with large weighted exponents all
one needs is to apply Theorem \ref{manifolds1} to the function 
$$ \Phi_{\pmb{s}}(\pmb{q}) \df \norm{\pmb{q}}_{\pmb{s}}^{1/n}
,$$ observing that 
one has 
 $$ \hat{\omega}_{\pmb{s}}^*(\pmb{\xi}) \df \sup\{ \gamma: \limsup_{t\to \infty}t^\gamma \psi_{ \Phi_{\pmb{s}},\pmb{\xi}}(t)
  < \infty
  \}.$$
This way we arrive at
 \begin{corollary}\label{manifolds4}    Let $S$ be as in Corollary
   \ref {manifolds2}, and let $\pmb{s}$ be as in  \eqref{tw1}. Then  
 there exist uncountably many  totally irrational  $\pmb{\xi}  \in  {S}$ such that 
$\hat{\omega}_{\pmb{s}}^*(\pmb{\xi}) = \infty$.
 \end{corollary}

 {Exact transference theorems 
 for the weighted setting were obtained quite recently. Improving on a paper 
 by Chow, Ghosh, Guan,  Marnat and Simmons \cite{many}, German
 \cite{GG} proved a transference inequality which in particular states
 that 
 $$
 \hat{\omega}_{\pmb{s}}^*(\pmb{\xi}) =\infty\,\,\,\,\,\,
 \Longrightarrow\,\,\,\,\,\,
\hat{\omega}_{\pmb{s}}^*(\pmb{\xi}) \ge   \frac{1}{n(1-\delta)},
$$
where $\delta$
 is defined in (\ref{tw2}). This leads to the following
 }
{\begin{corollary}\label{manifolds5}    Let ${S}$ be as in Theorem
    \ref {thm: abstract}, let $\pmb{s}$ be as in
    \eqref{tw1}, and let $\delta$ be as in \eqref{tw2}. Then  
 there exist uncountably many  totally irrational  $\pmb{\xi}  \in  {S}$ such that
{$\hat{\omega}_{\pmb{s}}(\pmb{\xi})\ge {\frac{1}{n(1-\delta)}}$}. 
\end{corollary}}


\subsection{Applications to manifolds and fractals}\label{mflds} 
We now describe two classes of subsets $S\subset \R^n$ for which  the
assumptions of Theorem \ref {thm: abstract} can be verified.
The first application involves certain product subsets of
$\R^n$. Recall that a subset of $\R$ is called {\sl perfect} if it is
compact and has no isolated points. 
\begin{theorem}\label{thm: fractals}
Let $n \geq 2$ and let $S_1, \ldots, S_n$ be perfect subsets
of $\R$ such that 
\eq{density}{\Q \cap S_k \text{ 
is dense in }S_k\text{  for each }k \in \{1,2\}.} Let $S = \prod_{j=1}^n S_j$.
Then there are collections $\{L_i\}, \,
\{L'_j\}, \, \{A_i\}$ satisfying the hypotheses of Theorem \ref{thm: abstract}.
In particular, the conclusions of Theorem
\ref {thm: abstract} and Corollaries
\ref{manifolds2}---\ref{manifolds5} hold for $S$. 
\end{theorem}


For example, the above theorem applies to products of one-dimensional limit sets of
rational iterated function systems such as the middle third Cantor set
and its generalizations. 
Thus as a special case we see
that a Cartesian product of two copies of Cantor's middle thirds set
contains uncountably many totally irrational singular vectors.
The question of determining the Hausdorff dimension of the set of
singular vectors in this fractal was raised in the recent paper
\cite{BCC} of Bugeaud, Cheung and Chevallier, and an upper bound was
obtained by Khalil \cite{Khalil}.

\smallskip

As a second application, let
us consider real analytic submanifolds.  
\begin{theorem}\label{manifolds1}  Let ${\mathcal{S}}$ be a connected real
  analytic submanifold  of $\R^n$ of 
dimension at least $2$  which is not contained in any proper rational affine
subspace of $\R^n$. Then there are collections $\{L_i\}, \,
\{L'_j\}, \, \{A_i\}$ satisfying the hypotheses of Theorem \ref{thm: abstract}.
In particular, the conclusions of Theorem
\ref {thm: abstract} and Corollaries
\ref{manifolds2}---\ref{manifolds5} hold for $\mathcal{S}$. 
\end{theorem}

\subsection{Optimality of exponents}\label{upper}
 One may wonder whether it is
  possible to strengthen the conclusion of Corollary \ref{manifolds3} 
  and, for $S $ 
 as in Theorem
    \ref {thm: abstract},
construct  totally irrational  $\pmb{\xi}  \in S$ {with $\hat{\omega}(\pmb{\xi}) = 1$},
thereby obtaining an optimal result  {identical} to the conclusion of
Theorem \ref{B} restricted to $S$. However this is not the
case.  
To explain why, we give two examples. First of all we refer to the
paper \cite{KM}, where it is shown that for hypersurfaces of the
  form 
$$\mathcal{S} = \{\pmb{\xi}\in \mathbb{R}^n:   f(\pmb{\xi}) = 1\}\subset \mathbb{R}^n,$$ where $f$ is a  homogeneous polynomial 
of degree $s$ such that 
$
{\#\{\pmb{x}\in\Q^n: f(\pmb{x} )= 0\} < \infty},
$
one has 
$$\sup_{\text{totally irrational }\pmb{\xi} \in \mathcal{S}} \hat{\omega}(\pmb{\xi}) \le H_{n-1,s},
$$
where $ H_{n-1,s}< 1$ is 
the unique positive root of the equation
$ 1-x =x\cdot \sum_{k=1}^{d} \left(\frac{x}{s-1}\right)^k$.
 In particular  for any totally irrational $ \pmb{\xi} $ on the unit sphere 
$$
 \{ 
 (x_1,\dots, x_n)
 :\,\,\,\,\, x_1^2+\dots+x_n^2 = 1\}\subset \R^n
$$
 one has
$ \hat{\omega}(\pmb{\xi}) \le H_{n-1}$,
where $H_{n-1} = H_{n-1,2}$ is the unique positive root of the polynomial \linebreak
$
x^n+
\dots+x - 1.
$
More general results for quadric hypersurfaces, as well as the
optimality of the aforementioned bound,
were very recently proved  
by
Po\"els and Roy 
in 
\cite{R1u}.

In addition to that, in \S \ref{BAD}  below we show that in the
case when $\mathcal{S}$ is a so-called badly approximable affine
subspace of $\R^n$,  the value 
$ \hat{\omega}(\pmb{\xi})$ is uniformly bounded away from $1$ for any
totally irrational  $\xi\in \mathcal{S}$.

 \section{Proof of Theorem \ref{thm: abstract}}
 The idea of proof goes back to Khintchine's original argument  \cite{H1}  and has
 appeared in many incarnations in work on the subject, see  \cite{m} for a survey.
We retain the notation and assumptions of the
 theorem;  {that is, 
 \begin{itemize}
 \item $S \subset \R^n$  is a nonempty locally closed subset; 
  \item  $\mathcal{L} \df
  \{L_1, L_2, \ldots\}$, $\mathcal{L}' \df
  \{L'_1, L'_2, \ldots\}$ are disjoint collections of distinct closed
  subsets of $S$ such that  conditions (a)--(d) of Theorem \ref{thm: abstract} hold;  \item  $\Phi : {\mathbb{Z}}^n{\nz}\to \mathbb{R}_+$ is such that \eqref{ccoo} holds;  \item $ \varphi:\mathbb{R}_+\to \mathbb{R}_+$ is non-increasing.
  \end{itemize}} Also for a rational affine hyperplane $A_i$ as in the
 statement {of the theorem} we let $\mathbf{m}_i \in \Z^{n+1}$ be a primitive vector so
 that $A_i = A_{\mathbf{m}_i}$, where the notation and normalization are
 as in \equ{eq: def hyperplane}. 


\begin{proof}[Proof of Theorem \ref{thm: abstract}]
Let
 $$\BB \df \left\{\pmb{\xi} \in S: \exists\, t_0 \text{ such that
}  \forall \,t \geq t_0, \, \psi_{\Phi, \pmb{\xi}}(t) \leq \varphi(t)
\text{ and } \pmb{\xi} \text{ is 
  totally irrational} \right\},$$
and suppose by contradiction that $\BB$ is at most countably infinite. Write 
  $\BB = \{\pmb{b}_1, \pmb{b}_2, \ldots \}$ (in case $\BB$ is finite,  this is
  a finite list). Let $\mathcal{W}$ be an open subset of $\R^n$ for
  which $S = \overline{S} \cap \mathcal{W}$.  {Put $\mathcal{U}_0 = \mathcal{W}$, $\pmb{q}_{0} = 0$, $p_{0} = 0$, $i_{0} = 0$, $\Phi(0) = 0$.}
    We will show that for each $\nu \in \N$
    there is
    a bounded open set $\mathcal{U}_\nu \subset \mathcal{W}$, and an index
    $i_\nu \in \N$, such that, with the notation  
    $$ (p_\nu, \pmb{q}_\nu)  \df   \mathbf{m}_{i_\nu},$$
    the following
    conditions are satisfied:  

\begin{enumerate}
\item  {$\varnothing \neq  \overline{S \cap \mathcal{U}_\nu} \subset \mathcal{U}_{\nu-1}$};
\item 
 {
 $i_\nu > i_{\nu-1}$ and  $\Phi(\pmb{q}_\nu) > \Phi(\pmb{q}_{\nu-1})$ for all $\nu\in\N$.}
\item
  For all $k < {\nu}$, $\mathcal{U}_\nu$ is disjoint from $L_k \cup
  L'_k \cup \{\pmb{b}_k \}$.
\item
  For all $\nu\in\N$ and all $\pmb{\xi} \in \mathcal{U}_\nu$ we have
  $$
  |\pmb{\xi} \cdot \pmb{q}_{\nu-1} - p_{\nu-1}| <
  \varphi\big(\Phi(\pmb{q}_\nu) \big). 
$$
\item
  For all $\nu\in\N$, ${\mathcal{U}_\nu \cap
  L_{i_\nu}} \neq \varnothing.$
\end{enumerate}

To see this suffices, take a point
\eq{eq: suffices}{\pmb{\xi} \in S \cap \bigcap_\nu \mathcal{U}_\nu = 
\bigcap_\nu \overline{S \cap\mathcal{U}_\nu}.}
This intersection is nonempty since the right-hand side of \equ{eq:
  suffices} is by (1) an 
intersection of nonempty nested compact sets, and the equality
between both sides of \equ{eq: suffices} follows from the fact that
for $\nu \geq 2$, the sets
$\overline{\mathcal{U}_\nu}$ are contained in $\mathcal{W}$. We will reach a contradiction by
showing that both $\pmb{\xi} \notin \BB$ and $\pmb{\xi} \in \BB$. By
(3), $\pmb{\xi}$ is not equal to any of the $\pmb{b}_i$ and hence $\pmb{\xi}
\notin \BB$. Also by 
(3), $\pmb{\xi}$ is {not} contained in 
any of the sets in the
collections $\mathcal{L}, \mathcal{L}'$, and thus by 
\equ{eq: Ls cover As}, 
$\pmb{\xi}$ is totally  
irrational. The function $\varphi$ is non-increasing by assumption,
and so is the irrationality measure
function $t \mapsto \psi_{\Phi, \pmb{\xi}}(t)$, as follows from its
definition (\ref{0p0}). {The properness condition \eqref{ccoo} guarantees that $\Phi(\pmb{q}_\nu)\to\infty$ as $\nu\to\infty$.}
By (2), for any $t>t_0 \df \Phi(\pmb{q}_1)$ there is $\nu$ with $t \in
\left[\Phi(\pmb{q}_\nu), \Phi(\pmb{q}_{\nu+1}) \right]$  and by (4)  we have
\[
    \psi_{\Phi, \pmb{\xi}}(t) \leq 
\psi_{\Phi, \pmb{\xi}}\big(\Phi(\pmb{q}_{\nu})\big) \leq 
\dz{\, 
  \pmb{q}_\nu \cdot \pmb{\xi}} 
\leq 
| \pmb{q}_\nu
\cdot \pmb{\xi} - p_\nu| < \varphi\big(\Phi(\pmb{q}_{\nu+1}) \big) \leq
\varphi(t).
\]
This shows that 
$\pmb{\xi} \in \BB$. 

Note that when utilizing the above properties, we did not require
property (5). However we will use it for constructing the sequences 
$\mathcal{U}_\nu, \, i_\nu$. 

\smallskip
The inductive
construction starts with $\nu=1$.
{Choose $i_1\df\min\{i\in\N: L_{i} \neq
\varnothing\}$, which exists in view of hypothesis (d),
and} define $\mathcal{U}_1$ to be some open set containing a point in $L_{i_1}$ and such that  $\overline{\mathcal{U}_1}\subset\mathcal{W}$.
Then {(1) and (5)} follow from this choice, and properties (2--4) hold 
vacuously for $\nu=1$.

Now suppose we have constructed $\mathcal{U}_k$ and $i_k$ with the
required properties for $k=1,
\ldots, \nu$, and we explain the construction 
for $\nu+1$. Let $i = i_\nu$. By (5) for $k =\nu$ we have $ \mathcal{U}_\nu \cap
L_{i} \neq \varnothing$. By hypothesis 
(b)  there is an infinite subsequence of indices $j$ such that along
this subsequence,
\eq{eq: nonempty intersection}{
\mathcal{U}_\nu \cap L_i \cap L_j \neq
\varnothing \ \text{ and } \ 
|A_j|\to_{j\to \infty} \infty.}
For each such $j$, write $A_j = A_{\mathbf{m}_j}, \, \mathbf{m}_j = (p'_j,
\pmb{q}'_{j})$. Then by \equ{eq: slightly}, along this subsequence
we have 
$\|\pmb{q}'_j\| \to \infty$, and hence by the property 
(\ref{ccoo}) of $\Phi$, we can 
choose ${j > i}$ 
so that 
$\Phi(\pmb{q}'_j ) >
\Phi(\pmb{q}_\nu)$. We 
then set $i_{\nu+1} = j.$ This choice
ensures that (2) holds for $\nu+1$.
Let
$$
\pmb{\xi}_1 \in \mathcal{U}_\nu \cap L_i \cap L_j. 
$$
The point $\pmb{\xi}_1$ belongs to $L_i$ and hence satisfies $\pmb{\xi}_1
\cdot \pmb{q}_{\nu} = 
p_{\nu} $. By continuity, we can take a small neighborhood
${\mathcal{V} \subset \mathcal{U}_\nu}$ around $\pmb{\xi}_1$, 
so that
for all $\pmb{\xi} \in {\mathcal{V}}$ we have
$$
|\pmb{\xi} \cdot \pmb{q}_{\nu} - p_{\nu}| <
  \varphi\big(\Phi(\pmb{q}_{\nu+1}) \big). 
  $$
  This is the inequality in (4), for $\nu+1$.
  
  Since $\pmb{\xi}_1 \in L_j = L_{i_{\nu+1}}$  we have ${\mathcal{V}} \cap
  L_{i_{\nu+1}} \neq 
  \varnothing$, so we can apply hypothesis (c) to find that there is
$$
\pmb{\xi} \in 
L_j \cap 
 {\mathcal{V}} \smallsetminus
\bigcup_{k < {\nu+1}} \big(L_k \cup 
L'_k \cup \{\pmb{b}_k\}\big). 
$$
Furthermore, we can take a small enough neighborhood
$\mathcal{U}_{\nu+1}$ of $\pmb{\xi}$ so that 
  $$\overline{\mathcal{U}_{\nu+1}} \subset 
   \mathcal{U}_\nu, 
\ \text{ and }  \ \mathcal{U}_{\nu+1} \cap \bigcup_{k < {\nu+1}} (L_k \cup
L'_k \cup \{\pmb{b}_k\}\big) = \varnothing. 
$$
With these choices $\mathcal{U}_{\nu+1}$ and $i_{\nu+1}$ will also
satisfy (1), (3) and (5). Thus we have completed the 
inductive construction.  
    \end{proof}


With Theorem \ref{thm: abstract} in hand, it is easy to complete the
\begin{proof}[Proof of Theorem \ref{thm: fractals}]
Recall that we are given $S = \prod_{j=1}^n S_j$, where $S_1, \ldots, S_n$ are perfect subsets
of $\R$ satisfying \equ{density}.
Let $\mathbf{e}_1, \ldots, \mathbf{e}_n$ be the standard base
vectors, and let
$\{A_i\} $ be the collection of all rational hyperplanes which are normal
to one of $\mathbf{e}_1 , \mathbf{e}_2$ and have nontrivial
intersection with $S$ (where each of the rational
hyperplanes appears exactly once). That is, each of the hyperplanes
$A_i$ is of the form
\eq{eq: for fractals}{
A_i = \left\{ \pmb{\xi} \in \R^{n} : \xi_{k_i} = \frac{p_i}{q_i} \right\},
\text{ where } p_i \in \Z, \ 
q_i \in \N \text{ are coprime, and } k_i \in \{1,2\};} 
note that necessarily we have $\frac{p_i}{q_i}\in S_{k_i}$.

For each $i$ define $L_i \df S \cap A_i$, and let $\{L'_j\}$ denote
the collection of non-empty intersections $S \cap A$, where $A$ is a rational
affine hyperplane, and the set $L'_j$ does not appear in the list
$\{L_i\}$. We claim that with these choices, hypotheses (a)---(d) of
Theorem \ref{thm: abstract} are satisfied. 

Indeed, (a) is obvious from the definition, and (d) follows from \equ{density}. For (b) and (c), suppose for concreteness that $k_i=1$. Then it follows from \equ{eq: for fractals} that
\eq{eq: description}{
  L_i =  \left\{\pmb{\xi}\in \R^{n}: \xi_{1} = \frac{p_i}{q_i} \text{ and } \xi_j \in S_j \  \forall\,j
  \neq 1 \right\}.}
Let $\pmb{\xi} \in L_i$ and let
$p_j/q_j$ be a sequence of distinct rationals in $S_2$ satisfying
$p_j/q_j \to \xi_2$. Such a sequence exists since $S_2$ is perfect and
the rationals are dense in $S_2$. Let  
$$L_j\df  \left\{ \pmb{\xi} \in \R^{n} : \xi_{2} = \frac{p_j}{q_j} \right\}\cap S.$$ 
Then it is clear from \equ{eq: description} that $L_i \cap L_j$
contains 
elements $\pmb{\xi}_j$ such that $\pmb{\xi}_j \to \pmb{\xi}$, and such that
$\pmb{\xi}_j$ differs
from $\pmb{\xi}$ only in the $2$nd 
coordinate. Also $|A_j|  = q_j \to \infty$
and so for any $T>0$, $\pmb{\xi}$ is 
an accumulation point of the sets $L_i \cap L_j$ with $|A_j|>T$. This proves
(b). To show (c), note that because   $S_2,\dots,S_n$ are perfect, the intersection of the set \equ{eq: description} with an arbitrary open subset of $\R^n$ cannot lie in a  union of  finitely many proper affine subspaces of $\R^n$ different from $A_i$; hence \equ{densityinL}.
\end{proof}

\section{Real  analytic submanifolds}\label{an}
Let $k \leq n$, and let $\mathcal{U} \subset \R^k$ be open. We say
that ${f}: \mathcal{U}\to \R^n$
is {\sl real analytic immersion} if it is injective, each of its
coordinate functions ${f}_i :
\mathcal{U} \to \R \ (i=1, \ldots, n)$ is infinitely differentiable, 
the Taylor series of each ${f}_i$ converges in a neighborhood of
every $\pmb{x} \in \mathcal{U}$,  
and the derivative mapping $d_{\pmb{x}} {f}: \R^k \to \R^n$ has
  rank $k$.
By a  {\sl $k$-dimensional real analytic submanifold in
  $\R^n$} we mean a subset $\mathcal{M} \subset \R^n$ such that for
every $\pmb{\xi} \in \mathcal{M}$ there is a neighborhood $\mathcal{V}
\subset \R^n$ containing $\pmb{\xi}$, an open set $\mathcal{U} \subset \R^k$, 
and a real analytic immersion ${f}: \mathcal{U} \to \R^n$ such that
$\mathcal{V} \cap \mathcal{M} = {f}(\mathcal{U}).$ By a real analytic
{\sl curve} (resp., {\sl surface}) we mean a connected one-dimensional (resp., {\sl two-dimensional}) real analytic
submanifold. 
 A mapping ${h} :
\mathcal{M} \to \R^m$ is {\sl real analytic} if for any $\pmb{\xi},
{f}, 
\mathcal{U}$ as above, each coordinate function of ${h} \circ
{f}: \mathcal{U} \to \R^m$ is infinitely differentiable and its Taylor series converges
in some neighborhood of ${f}^{-1}(\pmb{\xi})$.

The crucial property which distinguishes real analytic submanifolds from
smooth manifolds, and follows easily from definitions, is the
following. Let $\mathcal{M}_1, \mathcal{M}_2$ be real
analytic submanifolds (where we equip them with the topology inherited from the
ambient space $\R^n$). Then, if the intersection $\mathcal{M}_1 \cap
\mathcal{M}_2$ has nonempty interior in $\mathcal{M}_1$, then
this intersection is open in $\mathcal{M}_1$; and thus, if additionally
$\mathcal{M}_1$ is connected and $\mathcal{M}_2$ is closed, then
$\mathcal{M}_1 \subset \mathcal{M}_2$.

 A subset 
$\mathcal{N} \subset \mathcal{M}$  is called {\sl semianalytic} if it
is locally described by 
finitely many equalities and inequalities involving real analytic
functions, i.e. for every $\pmb{\xi}_0 \in \mathcal{N}$ there is an open
neighborhood $\mathcal{U}$ containing $\pmb{\xi}_0$ such that 
$$
\mathcal{N} \cap \mathcal{U} = \left\{{\pmb{\xi}} \in \mathcal{M} \cap
  \mathcal{U} : \forall i, \, 
{h}_i({\pmb{\xi}})=0 \text{ and } \forall j, \, \bar{{h}}_j({\pmb{\xi}}) >0
\right\},
$$
for finitely many real analytic functions ${h}_i, \bar{{h}}_j$ on
$\mathcal{M} \cap \mathcal{U}$. For background on the geometry of analytic and
semianalytic manifolds we refer the reader to \cite{BM} 
and the references therein. In particular the reader may consult
\cite{BM} for the definition of the {\sl dimension} of a semianalytic set. 

We will need to decompose
semianalytic subsets into analytic submanifolds. In this regard we
have the following (see \cite[\S 2]{BM}):

\begin{proposition}\label{prop: stratification}
Let $\mathcal{N} \subset \mathcal{M}$ be a semianalytic subset
of a real analytic submanifold $\mathcal{M} \subset \R^n$.  
Then any connected component of $\mathcal{N}$ is semianalytic, and
$\mathcal{N}$ has a locally finite presentation as  
a disjoint union of sets $\mathcal{N}_1,
\mathcal{N}_2, \ldots$, each of which is a connected
analytic submanifold of 
dimension at most $\dim \mathcal{N}$, and such that
\eq{eq: almost disjoint}{
i \neq  j, \ \mathcal{N}_i \cap \overline{\mathcal{N}_j} \neq \varnothing \
\implies \ \dim \mathcal{N}_j > \dim \mathcal{N}_i.
}
  \end{proposition}

  }

It will be easier to work with real analytic surfaces than
with manifolds of higher dimension. The reason for this is that in this
case it will be possible to describe a stratification as in
Proposition \ref{prop: stratification}  in topological terms. 
\begin{proposition}\label{prop: two dimensions simpler}
  Let $\mathcal{S}$ be a bounded real analytic surface, and
  let $A$ be an affine hyperplane such that $\mathcal{S} \not \subset
  A$. Denote by $F$ the set of points $\pmb{\xi} \in \mathcal{S} \cap A$ for which
  there does not exist a neighborhood $\mathcal{U}$ of $\pmb{\xi}$ such that
  $\mathcal{U} \cap \mathcal{S} \cap 
  A$ is a real analytic curve. Then $F$ 
  is finite, the number of connected components of $(\mathcal{S} \cap
  A) \smallsetminus F$ is finite, and each of these connected
  components is a real analytic curve. 
  \end{proposition}

We will refer to the connected components of the set $(\mathcal{S} \cap
  A) \smallsetminus F$ as the {\sl one-dimensional basic components of
    $\mathcal{S} \cap 
    A$}.


  \begin{example}\label{example: erroneously}
Let $n=3$, let $\mathcal{S}$ be 
  defined by
  $$\mathcal{S} = \left\{ \left(x,y, xy \right ) : x,y \in (-1,1) \right \},$$
    and let 
  $$A = \left\{\left(x, y, 0\right) : x,y \in \R\right \}.$$
  Then $A \cap \mathcal{S}$ is the union of a vertical line $\left\{x =
  0\right\}$ and a horizontal
  line  $\left\{y=
  0\right\}$ in the plane $A$, intersecting at the origin
$(0,0,0)$. The set  $F$ defined in Proposition 
  \ref{prop: two dimensions simpler} consists of the origin, and 
  the one-dimensional basic
  components are four open intervals (two horizontal and two vertical)
  in $A$. 
\end{example}

\begin{proof}[Proof of Proposition \ref{prop: two dimensions simpler}]
Write $\mathcal{S}_0 \df \mathcal{S} \cap A$, a semianalytic subset of
$\mathcal{S}$. Clearly $\dim
\mathcal{S}_0 \leq 2$ and we claim that $\dim
\mathcal{S}_0 \neq 2$. Indeed, if this were to hold then $
\mathcal{S}_0$ would be open in $\mathcal{S}$, but also closed since
$A$ is a closed subset of $\R^n$. By connectedness this would imply
$\mathcal{S} \subset A$, contrary to assumption. 

Thus $\dim \mathcal{S}_0 \leq 1$. We treat separately the cases $\dim
\mathcal{S}_0 = 0$ and $\dim
\mathcal{S}_0 =  1$. If $\mathcal{S}$ has dimension $0$, then each
of its connected components is a real analytic submanifold of
dimension $0$, i.e. $\mathcal{S}_0$ is a discrete subset of
$\mathcal{S}$. Moreover $\mathcal{S}_0$ is finite, since the collection described in
Proposition \ref{prop: stratification} is locally finite and
$\mathcal{S}$ is bounded, and by definition $F = \mathcal{S}_0$. 

If $\dim \mathcal{S}_0 =1$, then by
Proposition
\ref{prop: stratification} (and using again that $\mathcal{S}$ is
bounded) we can write
$\mathcal{S}_0$ as a disjoint union
$F_0 \cup F_1$, where $F_0$ is a finite set of points and $F_1$ is a
finite union of disjoint real analytic curves $\mathcal{N}_i$. 
Such a
stratification is not unique, but we choose one so that the
cardinality of $F_0$ is as small as possible. We claim that with this
choice, $F_0 = F$ and the real analytic curves $\mathcal{N}_i$ are the connected
components of $\mathcal{S}_0 \smallsetminus F$. 

To see this, note that since the $\mathcal{N}_i$ are real analytic
curves, any point in any one of the $\mathcal{N}_i$ 
cannot belong to the set $F$, so $F \subset F_0.$ Suppose if possible
that there is some ${\pmb{\xi}} \in F_0 
\smallsetminus F$. 
Since ${\pmb{\xi}} \notin F$, it is not an
isolated point of $\mathcal{S}_0$. Thus, if we denote by
$F_1({\pmb{\xi}})$  the 
collection of curves
$\mathcal{N}_i$ for which ${\pmb{\xi}} \in
\overline{\mathcal{N}_i}$, then $F_1({\pmb{\xi}} ) \neq
\varnothing$.

Let $\eta$ be the connected component of $\mathcal{S}_0 \smallsetminus
F$ containing ${\pmb{\xi}}$. Then $\eta$ is a real analytic curve. By
the connectedness of $\eta$ and property \equ{eq: almost disjoint}, any
$\mathcal{N}_i$ in the collection $F_1({\pmb{\xi}} )$ must be contained in
$\eta$. Since ${\pmb{\xi}}$ is a smooth point of $\eta$, 
i.e.\ there is a neighborhood $W$ of $\xi$ such that $W \cap \mathcal{S}_0  = W \cap \eta$,
it follows that $F_1({\pmb{\xi}})$ consists of two real analytic curves $\mathcal{N}_i,
\mathcal{N}_j$ such that the union $\gamma \df \mathcal{N}_i \cup \{{\pmb{\xi}}\} \cup
\mathcal{N}_j$ is  
also a real analytic curve contained in $\eta$.
We can therefore modify 
$F_0$ and $F_1$, by replacing $F_0, \, F_1 $  respectively with
$$F_0 \smallsetminus
\{{\pmb{\xi}}\} \ \ \text{ and } \ F_1 \cup \{\pmb{\xi}\} = F_1 \cup \gamma
\smallsetminus \left(\mathcal{N}_i \cup \mathcal{N}_j\right).$$
But
this contradicts the minimality of $F_0$, showing that $F_0 =
F$. Since by \equ{eq: almost disjoint} any boundary point of any
$\mathcal{N}_i$ is in $F$, the $\mathcal{N}_i$ are open and closed as
subsets of $\mathcal{S}_0 \smallsetminus F$. Thus they coincide with
the connected components of $\mathcal{S}_0 \smallsetminus F$.
  \end{proof}
\medskip

In order to be in a position to apply Proposition \ref{prop: two dimensions simpler} we will
need the following.

\begin{proposition}\label{prop: reducing dimensions}
Let $k \geq 2$, and let $\mathcal{M} \subset \R^n$ be a connected $k$-dimensional
real analytic submanifold which
is not contained in a proper rational affine subspace of $\R^n$. Then
$\mathcal{M}$ contains a   bounded real analytic surface
which is not contained in a proper rational affine subspace of
$\R^n$. 
  \end{proposition}

Proposition \ref{prop: reducing dimensions} is proved by 
induction on the dimension $k$, the base case $k=2$ being obvious. For
$k \geq 3$, the
deduction of case $k$ from case $k-1$ follows from
the observation that any proper affine rational subspace of $\R^n$ is contained
  in a  rational affine hyperplane, 
and from the following. For each $k$ we denote $I_k \df (0,1)^k$ the
open unit $k$-dimensional cube. 

\begin{lemma}\label{wW}
Suppose that  for $ k \ge 3$, $\mathcal{M}$ is a $k$-dimensional  real
analytic submanifold which is the
image of $I_k$ under a real analytic
immersion $f:I_k \to \R^n$. Suppose also that $\mathcal{M}$
is not contained in any rational affine hyperplane. Then
there exists $\alpha \in (0,1)$ such that the 
analytic manifold $f_\alpha(I_{k-1})$, where
\begin{equation}\label{alla}
 f_\alpha: I_{k-1} \to \R^n, \ \  f_\alpha  (x_1,...,x_{k-1}) \df f(x_1,...,x_{k-1},\alpha),
\end{equation}
does not belong to any rational affine hyperplane.

\end{lemma}

\begin{proof}[Proof of Lemma \ref{wW} (and hence of Proposition
  \ref{prop: reducing dimensions})]
  If 
  the conclusion of
  the Lemma is not true, then  for any 
$\alpha \in (0,1)$ there exists a  rational affine hyperplane $A$ containing 
the image of the function (\ref{alla}).
This means that
\begin{equation}\label{nill}
\bigcup_{\mathbf{m}} f^{-1}\left (A_{\mathbf{m}}\cap \mathcal{M} \right)
 =I_k
\end{equation}
 where the union is taken over all primitive vectors $\mathbf{m} \in
 \Z^{n+1}$ and $A_{\mathbf{m}}$ is the rational affine hyperplane
 defined via \equ{eq: def hyperplane}. This is a countable union of
 closed subsets of $I_k$ so by the Baire category theorem, one of them
 has nonempty interior. That is there is a nonempty open subset $\mathcal{U}
 \subset I_k$ and $\mathbf{m}_0$ such that $f(\mathcal{U}) \subset
 A_{\mathbf{m}_0}$. That is, $\mathcal{M} \cap
 A_{\mathbf{m}_0}$ has nonempty interior in $\mathcal{M}$. By analyticity
 and connectedness of $\mathcal{M}$ we then have $\mathcal{M} \subset
 A_{\mathbf{m}_0}$, contrary to hypothesis. 
\end{proof}

\section{Proof of Theorem \ref{manifolds1}}\label{pf}
We first explain informally the main difficulty in the proof and the idea that
allows us to overcome it. As 
was mentioned above, the intersections of real analytic submanifolds 
with
affine hyperplanes 
are semianalytic sets,  but they need not themselves be real
analytic submanifolds. This makes it tricky to verify the hypotheses
of Theorem \ref{thm: abstract}. To deal with this, we first pass to
the case in which $\mathcal{S}$ is a surface and is not contained
in a proper rational affine hyperplane. This means that the
intersections $\mathcal{S} \cap A$ can be described by Proposition \ref{prop: two
dimensions simpler}. Moreover for some affine hyperplanes $A$, the
sets $\mathcal{S} \cap A$ can be taken to satisfy a
transversality condition which implies that they {\it are}  real analytic
curves. Specifically in the proof below, the $L_i$ will be closed
real analytic curves, while the $L'_j$ will be basic one-dimensional
components of one-dimensional semianalytic sets. We now proceed to the details of the
argument. 

\begin{proof}[Proof of Theorem \ref{manifolds1}] By Proposition \ref{prop: reducing dimensions} we can assume that $\mathcal{S}$ is bounded,
connected  and two-dimensional. 
Let $\{A_i\}$ be the collection as in the proof of Theorem \ref{thm: fractals},
that is the $A_i$ are the affine rational hyperplanes normal to
one of the two standard basis vectors $\mathbf{e}_1, \mathbf{e}_2$.
For each $\pmb{\xi} \in \mathcal{S}$, the
tangent space $T_{\pmb{\xi}} \mathcal{S}$ is a two dimensional affine subspace of $\R^n$
passing through ${\pmb{\xi}}$. {Recall that $ \mathcal{S}$ is not contained in any proper rational affine
subspace of $\R^n$}; thus, by possibly replacing $\mathcal{S}$
with its smaller connected open subset, we can assume that for every ${\pmb{\xi}} \in
\mathcal{S}$, the tangent space $T_{\pmb{\xi}} \mathcal{S} $ is not normal to
either of $\mathbf{e}_1, \mathbf{e}_2$. This implies that we can view
$\mathcal{S}$ as a graph of a smooth function over its projection to the
two-dimensional 
space $V_{12} \df \spa (\mathbf{e}_1,
\mathbf{e}_2) \cong \R^2$. This implies furthermore that for each $i$
and each ${\pmb{\xi}}
\in \mathcal{S} \cap A_i$, the intersection $T_{\pmb{\xi}} \mathcal{S} \cap A_i$
is a transversal intersection, that is, an affine subspace of dimension
one. By taking $\mathcal{S}$ smaller, we can ensure that its projection to the plane $V_{12}$ is an open bounded convex set. Now define $L_i \df S \cap A_i$ (where we only take those indices $i$ for which $L_i$ is not empty). Each $L_i$ is closed as a subset of $\mathcal{S}$, and by the implicit function theorem it is a real analytic curve. Since the projection of $\mathcal{S}$ on $V_{12}$ is convex, each $L_i$ is also connected.

Having defined the collections $\{L_i \}, \{A_i\}$ we now define the
collection $\{L'_j\}$. For any rational affine hyperplane $A$ for
which $\mathcal{S} \cap A$ is nonempty, we have by Proposition
\ref{prop: two dimensions simpler} 
its one-dimensional basic components. There are at most countably many
such sets $\{\gamma_j\}$ where each $\gamma_j$ is a real analytic
curve whose closure $\overline{\gamma_j} $ satisfies that 
$\overline{\gamma_j} \smallsetminus \gamma_j$ consists of at most two
points. We take
$$\{L'_j\} \df \{\overline{\gamma_j} : \forall\, i, \, \gamma_j \not
\subset L_i\}. $$
We claim that with these choices, conditions (a)---(d) of Theorem
\ref{thm: abstract} are satisfied (note that as a real analytic
submanifold of $\R^n$, $\mathcal{S}$ is locally closed). 

Properties (a), (b), (d) are straightforward. 
Indeed, since each $L_i$ is a connected real analytic curve, the condition
$\gamma_j \not \subset L_i$ is equivalent to $\overline{\gamma_j} \not
\subset L_i$. Also the sets $\overline{\gamma_j}$ contain all points
of $\mathcal{S}$ which belong to rational affine hyperplanes but not
to one of the hyperplanes $A_i$. Thus we have (a). For (d) note we can apply the
projection to the plane $V_{12}$, since $\mathcal{S}$ is a graph over this plane. By
construction, the
projections of the $L_i$ form a dense collection of horizontal lines and a dense
collection of vertical lines. In particular (d) holds. 
For (b) we 
continue to work in the plane $V_{12} $. For every point ${\pmb{\xi}}$ on (say)
a horizontal 
line $\ell \subset V_{12}$, which is the projection of some $L_i$, there is a
sequence of intersection points ${\pmb{\xi}}_j$ of 
$\ell$ with vertical lines such that ${\pmb{\xi}}_j \to {\pmb{\xi}}$ and ${\pmb{\xi}}_j$ is
contained in spaces $A_j$. A computation similar to the one used in
the proof of Theorem \ref{thm: fractals} shows that along this
sequence, we have $|A_j| \to \infty,$ and (b) follows. 

For (c) we argue as follows. Let $F, F'$ be as in statement (c). The
set $L_i$ is a real analytic curve and for
each $k \in F$, $L_i \cap L_k$ is either empty or consists of a single
point. Now let $k' \in F'$, and suppose by contradiction that $L'_{k'}
\cap L_i$ has nonempty interior, relative to the topology on
$L_i$. Then, since $L'_{k'}$ is the closure of a real analytic curve
$\gamma$ with $\overline{\gamma} \smallsetminus \gamma$ consisting of
at most two points, $L_i \cap \gamma$ also has nonempty
interior relative to $L_i$, and, since the dimensions are both equal to
one, $\gamma \cap L_i$ also has nonempty interior relative to
the topology of $\gamma$. Since $L_i$ is closed and $\gamma$ is
connected, this means that $\gamma \subset L_i$, contradicting the
definition of $L'_{k'}$. 
\end{proof}

We close the section by commenting on Theorem \ref{C} and its proof given in \cite{KW},
which, as was mentioned in the introduction, 
contained an error. Since in this 
paper we prove a strengthening, namely Theorem~\ref{manifolds1}, we do not rewrite
the proof of \cite{KW} completely. Rather we explain the gap in the proof
and sketch how it can be fixed.

Theorem \ref{C} is derived in \cite{KW} from an abstract result
\cite[Theorem 5.1]{KW}, 
which is similar to Theorem \ref{thm: abstract} (abstracting
Khintchine's classical argument). The statement of 
\cite[Theorem 5.1]{KW}
involves two countable lists $X_1, X_2, \ldots$ 
and  $X'_1, X'_2, \ldots$ of closed subsets of a subset $X$ of a Lie
group. The application deals with a real analytic submanifold
$\mathcal{S} \subset \R^n$ of 
dimension at least two, embedded in the group $\SL_{n+1}(\R)$. In
order to conclude that $\mathcal{S}$ 
contains totally irrational singular vectors $\pmb{\xi}$, and to obtain a
bound on their associated parameter $\hat{\omega}(\pmb{\xi}) $, some
conditions on the sets $X_i, X'_j$ must be checked.
One of these is the following transversality condition:
\eq{eq: condition to be checked}{
  \text{ for every } i, j,\  X_i = \overline{X_i \smallsetminus X'_j}
}(which is analogous to hypothesis (c) of Theorem \ref{thm: abstract}). 
  The argument given in \cite{KW} defines the $X_i, X'_j$ as connected
  components of the 
  intersection of $\mathcal{S}$ with rational affine hyperplanes.
  It is then erroneously claimed 
  that \equ{eq: condition to be checked} holds for these
  choices. Indeed, with the notations of Example \ref{example:
    erroneously}, set $A_1 \df A$ and 
  $
  A_1' \df
  \left\{\left(0 , y,z \right ) : y,z \in \R \right\}.$
 Then $X_1 = \mathcal{S} \cap A_1$ is the union of two lines intersecting at
  a point, and
  $X'_1 = \mathcal{S} \cap A_1'$ is one of these lines.  
  So \equ{eq: condition to be checked} fails.
  
  It is possible to rectify the proof by adapting some of the
  arguments we used in the proof of Theorem \ref{manifolds1}; namely,
  by replacing $\mathcal{S}$ with a two-dimensional real analytic submanifold,
  and adjusting the definitions of the sets $X_i, X'_i$ using the
  notion of basic components. We leave the details to the reader.

\section{Badly approximable subspaces}\label{BAD}
In this section we  give upper bounds for the exponent
$\hat{\omega}(\pmb{\xi}) $ for points $ \pmb{\xi} \in A$  in  
case when an $s$-dimensional affine subspace $A$  of $\mathbb{R}^n$   is badly approximable.
 To define the latter property, we identify $\mathbb{R}^n$ with the affine subspace
$$
\mathbb{R}^n_1 :=\big\{ \mathbf{x} = (x_0,x_1,\dots,x_n) \in \mathbb{R}^{n+1}: \, x_0 = 1\big\}
$$
and consider the affine subspace
$$
\mathcal{A} \df \{ \mathbf{x} 
 \in \mathbb{R}^{n+1}: \,
x_0 = 1\, , 
(x_1,...,x_n) \in A
\}.
$$
Let us define the linear subspace
$$
L_A \df  {\rm span}\, \mathcal{A} \subset \mathbb{R}^{n+1}.
$$
It is clear that
 $
L_A $
has dimension 
$s+1$.
From Minkowski's convex body theorem it follows that there exists a constant $C$ dependent only on $n$  such that for any $A$ there exists infinitely many 
integer vectors $\mathbf{m}\in \mathbb{Z}^{n+1}$
such that 
$$
{\rm dist}\, (
L_A,\mathbf{m}
)^{n-s} \cdot \|\mathbf{m}\|^{{s+1}}<C.
$$
We define $A$ to be a {\it badly approximable} subspace if  
$$
\inf_{\mathbf{m}\in \mathbb{Z}^{n+1} \nz
}
\,\,\,
 {\rm dist}\, \left(
L_A, \mathbf{m}
\right) ^{n-s} \cdot \|\mathbf{m}\|^{{s+1}}
>0.
$$
It is clear that badly approximable subspaces exist. Moreover
from a famous theorem of Schmidt \cite{SL}
 it follows that they form a {\it thick} set (that is, the set of badly approximable subspaces in any non-empty open subset of the Grassmanian of  all  $s$-dimensional affine subspaces of $\R^n$   has full \hd). 
{Indeed, without loss of generality one can parametrize $A$ in the following form:
\begin{equation}\label{param}
A = \left\{\pmb{\xi} = \begin{pmatrix}\pmb{x}\\ \pmb{y}_0 + Y\pmb{x}   \end{pmatrix}: \pmb{x} \in\R^s\right\},
\end{equation}
where $Y\in M_{n-s,s}$ and $\pmb{y}_0\in\R^{n-s}$. 
Define 
$$
w_{s,n} \df \frac{s+1}{n-s}.
$$
Then it is easy to see that $A$ is badly approximable if and only if the augmented matrix \begin{equation}\label{aug}\widetilde Y := \left[\begin{matrix} \pmb{y}_0 & Y\end{matrix}\right]\in M_{n-s,s+1}\end{equation} is badly approximable, that is, if $$\inf_{\mathbf{q}\in\Z^{s+1}\nz}\|\mathbf{q}\|^{w_{s,n}} \|\langle \widetilde Y \mathbf{q}\rangle\| > 0.$$
Note} that
\eq{lessthan}{w_{s,n}<1\quad\Longleftrightarrow\quad s < \frac{n-1}{2}.}
Let
$W_{s,n}$ be the unique root of the equation
\begin{equation}\label{0W}
 x^{n+1}-w_{s,n}^{n-1}(1+
w_{s,n})x + w_{s,n}^n=0
\end{equation}
in the interval $(0,w_{s,n})$.

\begin{proposition}\label{prop: 4.1}
Let $A$ be an $s$-dimensional \ba\ affine subspace   of $\mathbb{R}^n$. Then:

\begin{itemize}
\item[\rm (i)]
for any $\pmb{\xi} \in A$ one has
\begin{equation}\label{BAD1}
\hat{\omega}(\pmb{\xi}) \le w_{s,n};
\end{equation}

\item[\rm (ii)]
for any  totally irrational $\pmb{\xi} \in A$  
one has
\begin{equation}\label{BAD2}
\hat{\omega}(\pmb{\xi}) \le W_{s,n}.
\end{equation}
\end{itemize}
\end{proposition}

\begin{remark}
In view of \equ{lessthan}, when $\pmb{\xi}$ is totally irrational, the estimate in (i) is non-trivial only if $s < \frac{n-1}{2}$.
If $s$ is fixed, then $w_{s,n} = O(\frac{1}{n})$ as $n\to \infty$, which
shows that for large $n$
the conclusion of Corollary \ref{manifolds3}, that is, the  existence of uncountably many  totally irrational  $\pmb{\xi}  \in {S}$ with
 $\hat{\omega}(\pmb{\xi})\ge \frac{1}{n-1}$, is close to
optimal, in some sense. 
Statement  (ii) gives a slight improvement
of this bound. 

In particular, we can consider the following examples: 

\begin{itemize}
\item[1)]  if
$n=4$ and $s=1$, the inequality  (\ref{BAD1}) gives $\hat{\omega} \le w_{1,4} = \frac{2}{3}$;

\item[2)]   if $ n=2$ and $s=1$, we have $ w_{1,2} = 2$. Equation 
(\ref{0W}) is now 
$x^3-6x+4 =0$,  and the  inequality  (\ref{BAD2}) gives
$$
\hat{\omega} \le 
W_{1,2} = \sqrt{3} -1 = 0.732... 
$$

\item[3)]   if  $ n=3$ and $s=1$, we have $ w_{1,3} = 1$,  equation 
(\ref{0W}) has the form 
 $x^4-2x+1 =0$, and the  inequality  (\ref{BAD2}) gives
$$
\hat{\omega} \le W_{1,3} = 0.54... 
$$

\item[4)]     if $ n=3$ and $s=2$, we have $w_{2,3} = 3$, equation 
(\ref{0W}) is now 
 $x^3-36x+27 =0$,  and the  inequality  (\ref{BAD2}) gives
$$
\hat{\omega} \le W_{2,3} = 0.759... 
$$
\end{itemize}

\end{remark}

\begin{proof}[Proof of Proposition \ref{prop: 4.1}]
{To prove (i), 
 we will use the following elementary
 \begin{lemma}\label{BAelementary} 
If
$\pmb{\xi} \in A$
and $A$ is badly approximable, then  there exists a positive $c$ such
that for every $q\in \N$ we have 
\begin{equation}\label{BADA}
\|\langle q\pmb{\xi}\rangle \| \ge  
c q^{-w_{n,s}}.
\end{equation}
\end{lemma}
From this lemma 
\eqref{BAD1} follows immediately.
 \begin{proof}[Proof of Lemma \ref{BAelementary}]
 We will use the parameterization \eqref{param}. Assume the contrary, that is, for some $\pmb{\xi} = \begin{pmatrix}\pmb{x}\\ \pmb{y}_0 + Y\pmb{x}   \end{pmatrix}$ and any $\varepsilon > 0$ there exists $\pmb{m} = \begin{pmatrix}\pmb{p}\\ \pmb{r}   \end{pmatrix}\in\Z^n$ such that 
 $$
 \|q\pmb{\xi} - \pmb{m}\| = \left\|\begin{pmatrix}q\pmb{x}\\ q\pmb{y}_0 + Y(q\pmb{x})   \end{pmatrix} - \begin{pmatrix}\pmb{p}\\ \pmb{r}   \end{pmatrix}\right\| < \varepsilon q^{-w_{n,s}}.$$ 
 In particular, $\|q\pmb{x} - \pmb{p}\| < \varepsilon q^{-w_{n,s}};$ thus, if we define $\mathbf{q} := \begin{pmatrix}q\\ \pmb{p}  \end{pmatrix}$, it follows that 
 \begin{equation}\label{normbound}\|\mathbf{q}\|\le C q\quad\text{for some $C = C(\pmb{x})$ independent on $q$, $\pmb{p}$ and }\varepsilon.\end{equation}
{Note that
 $$
 q\pmb{y}_0 + Y(q\pmb{x}) - \pmb{r} =  q\pmb{y}_0 + Y(\pmb{p} + q\pmb{x} - \pmb{p}) - \pmb{r} = \widetilde Y\mathbf{q} - \pmb{r}  + Y (q\pmb{x} - \pmb{p}),
 $$
 where $\widetilde Y$ is as in \eqref{aug}. Hence
 $$
 \|\langle \widetilde Y \mathbf{q}\rangle\| \le  \|q\pmb{y}_0 + Y(q\pmb{x}) - \pmb{r} \| + \|Y (q\pmb{x} - \pmb{p})\| < \widetilde C \varepsilon q^{-w_{n,s}},
 $$
 where $\widetilde C$ is a constant depending only on $Y$. Since $\varepsilon$ was arbitrary and in view of \eqref{normbound}, this shows that $\widetilde Y$, and hence the subspace $A$, is not badly approximable.}\end{proof}}

\ignore{  Suppose that  is very easy. We observe that 
if
$\pmb{\xi} =(\xi_1,...,\xi_n)\in A$
and $A$ is badly approximable, then \comm{\it (DK: sorry, maybe I am too slow, but I do not understand this implication. Can one of you perhaps prove this observation, relating $\|\langle q\pmb{\xi}\rangle \|$ to ${\rm dist}\, (
L_A,\mathbf{m}
)$?)} there exists a positive $c$ such
that for every $q\in \mathbb{Z}_+$ we have 
\begin{equation}\label{BADA}
\|\langle q\pmb{\xi}\rangle \| \ge  
c q^{-\frac{s+1}{n-s}},
\end{equation}
and so 
(\ref{BAD1}) follows.}
\medskip

 To prove (ii)
 we consider the ordinary Diophantine exponent
 $\omega = \omega(\pmb{\xi})$, defined as supremum of those $\gamma>0$ for which the inequality
 $$
||\langle q\pmb{\xi}\rangle ||< q^{-\gamma}
 $$
 has infinitely many solutions in $ q\in \N$.
 It is clear from   \eqref{BADA} that for $\pmb{\xi }\in A$ one has 
 \begin{equation}\label{BADB}
 \omega(\pmb{\xi})\le  \frac{s+1}{n-s} = w_{s,n}.
\end{equation}
 
 Then we apply the inequality
 \begin{equation}\label{pp}
  \frac{\omega(\pmb{\xi})}{\hat{\omega}(\pmb{\xi})}\ge {G_{n}},
   \end{equation}
   where 
   $G_{n}$ the unique positive root of the equation
  $$
   x^{n-1} =\frac{\hat{\omega}}{1-\hat{\omega}} (x^{n-2} + x^{n-3} + \dots+ x+1),
$$
which is valid for all totally irrational $ \pmb{\xi}\in \mathbb{R}^n$.
This result was proven in \cite{MaMo}, and a short and beautiful proof
was given recently in \cite{Ng}. 

Note that as $\frac{1}{n}\le \hat{\omega} \le 1,$ we have $G_n \ge 1$
and $G_n=1$ if and only if $\hat{\omega} = \frac{1}{n}$. 
We see that $G_n$ is also a root of the simpler equation
$$
(1-\hat{\omega} )x +\frac{\hat{\omega}}{x^{n-1}}= 1
\,\,\,
\Longleftrightarrow
\,\,\, g(x) \df 
(1-\hat{\omega}) x^n - x^{n-1}+\hat{\omega}= 0.
 $$
The polynomial $  g(x)$ in the interval $ x\ge1$ has the unique root $G_n$. Hence
$$x\ge G_n \Longleftrightarrow  (1-\hat{\omega}) x^n - x^{n-1}+\hat{\omega} \ge 0.
$$
Moreover,  since 
$$
\max_{\frac{1}{n}\le z\le 1}
\left( z(1-z)^{n-1}\right) = \frac{1}{n}\cdot \left( 1-\frac{1}{n}\right)^{n-1},
$$
we have
$$
\hat{\omega}(1-\hat{\omega})^{n-1}\le \frac{1}{n}\cdot \left( 1-\frac{1}{n}\right)^{n-1}.
$$
Therefore
$g\left(\frac{1}{1-\hat{\omega}} \cdot \frac{n-1}{n}
\right)
\le 0,
$
and we deduce that 
$G_n\ge \frac{1}{1-\hat{\omega}} \cdot \frac{n-1}{n}$.

Now (\ref{pp}) leads to
$$
(1-\hat{\omega}) \omega^n - \omega^{n-1}\hat{\omega} +\hat{\omega}^{n+1}\ge 0.
$$
The function
$$
f(y) \df 
(1-\hat{\omega}) y^n - y^{n-1}\hat{\omega} +\hat{\omega}^{n+1}
$$
increases when $  y\ge \frac{\hat{\omega}}{1-\hat{\omega}} \cdot \frac{n-1}{n}$.
Since $ \omega\ge \hat{\omega}\cdot G_n \ge
\frac{\hat{\omega}}{1-\hat{\omega}} \cdot \frac{n-1}{n}, $ we  deduce from (\ref{BADB})  the inequality
$$
(1-\hat{\omega}) w_{s,n}^n - w_{s,n}^{n-1}\hat{\omega} +\hat{\omega}^{n+1}\ge 0
$$ 
which immediately gives (\ref{BAD2}).
\end{proof}

\ignore{ \begin{remark}
$G_{n}$ is a root of the simpler equation
$$
(1- \hat{\omega})x +\frac{\hat{\omega}}{x^{n-1}} = 1,
$$
or  equation
$$
(1- \hat{\omega})x^n -{x^{n-1}} +\hat{\omega}=0.
$$
\end{remark}}

\subsection*{Acknowledgements}
The authors were supported by 
NSF grant DMS-1600814,
 RFBR grant No.~18-01-00886 and BSF grant 2016256, respectively. This work was started
during the second-named author's visit to Brandeis University, whose
hospitality is gratefully acknowledged.  The authors are also grateful
to the MATRIX institute for providing a stimulating atmosphere for a
productive collaboration, and to the anonymous referee for taking a
close look at the paper and catching some errors.



\begin{thebibliography}{BGSV16}




\normalsize
\baselineskip=17pt

 
 

\bibitem{Ba1} R.\,C.\ Baker, \textit{Metric diophantine
approximation on manifolds}, J.\ Lond.\ Math.\ Soc. (2) {\bf 14}
(1976), 43--48.



\bibitem{Ba2} \bysame, \textit{Dirichlet's theorem on
diophantine approximation}, Math.\ Proc.\ Cambridge Phil.\ Soc.\ {\bf 83}
(1978), 37--59.





   
   \bibitem{BD}
V.\,I. Bernik and M.\,M. Dodson,
\textit{Metric {D}iophantine approximation on manifolds}, volume 137 of
  {Cambridge Tracts in Mathematics},
Cambridge University Press, Cambridge, 1999.

\bibitem{BM}  E.\ Bierstone and P.\,D.\ Milman, \textit{Semianalytic and subanalytic sets},
Inst.\ Hautes \'Etudes Sci.\ Publ.\ Math.\ {\bf 67} (1988), 5--42. 

  
\bibitem{Bu} Y.\ Bugeaud,  \textit{Approximation by algebraic
integers and Hausdorff dimension},
J.\ London Math.\ Soc.\ (2) {\bf 65} (2002), no.\ 3, 547--559.

\bibitem{BCC} Y.\ Bugeaud, Y.\ Cheung, and N.\ Chevallier, \textit{Hausdorff
    dimension and uniform exponents in dimension two}, Math.\
  Proc.\ Cambridge Phil.\ Soc.\ {\bf 167} (2018), no.\ 2, 249--284.

  
  \bibitem{C}
J.\,W.\,S.\ Cassels, \textit{An Introduction to Diophantine Approximation}, Cambridge Tracts
in Math.\ and Math.\ Phys.\ {\bf 45}, Cambridge University Press,  Cambridge, 1957.




  \bibitem{CC}
Y.\ Cheung and N.\ Chevallier,
\textit{Hausdorff dimension of singular vectors}, Duke Math.\ J.\
{\bf 165}
(2016), no.\ 12, 2273--2329. 

\bibitem{many}
S.\ Chow, A.\ Ghosh, L.\ Guan, A.\ Marnat, and D.\ Simmons, \textit{Diophantine transference inequalities: weighted, inhomogeneous, and intermediate exponents},
    preprint available at 
{\tt arXiv:1808.07184} (2018).

\bibitem{d} S.\,G.\ Dani, \textit{Divergent trajectories of flows on homogeneous spaces and Diophantine approximation}, J.\ Reine Angew.\ Math. 359 (1985), 55–89.

\bibitem{DFSU}
T.\ Das, L.\ Fishman, D.\ Simmons, and M.\ Urba\'nski,
\textit{A variational principle in the parametric geometry of numbers},
preprint available at 
{\tt arXiv:1901.06602} (2019).

  \bibitem{DS} H.\ Davenport and W.\,M.\ Schmidt, \textit{Approximation to real numbers by algebraic integers}, Acta Arith.\ {\bf 15} (1969), 393--416.
 
 \bibitem{Davenport-Schmidt}   \bysame,
\textit{Dirichlet's theorem on diophantine approximation}, in: Symposia
Mathematica, Vol.~IV (INDAM, Rome, 1968/69),  1970, pp.\ 113--132.

 



\bibitem{DRV} M.\ Dodson, B.\ Rynne, and J.\ Vickers,
    \textit{Dirichlet's theorem and Diophantine approximation on manifolds},
J.\ Number Theory {\bf 36} (1990), no.\ 1, 85--88.


  

\bibitem{GG}
O.\,N.\ German, \textit{Transference theorems for Diophantine approximation with weights},
 preprint available at 
 {\tt arXiv:1905.01512}  (2019).



 
 \bibitem{tb}
V.\ Jarn\'{\i}k, \textit{Zum Khintchineschen «\"{U}bertragungssatz»}, Trav.\ Inst.\ Math.\ Tbilissi  {\bf 3} (1938),
193--212.
 
\bibitem{J}
 \bysame, \textit{Eine Bemerkung \"{u}ber diophantische Approximationen}, Math.\ Z.\ {\bf 72}, no.\ 1 (1959), 187--191.

\bibitem{Khalil} O.\ Khalil, \textit{Singular vectors on fractals and
    projections of
    self-similar measures,} preprint available at 
  {\tt arxiv.org/abs/1904.11330
    } (2019). 


  
   \bibitem{H1} A.\,Ya.\ Khintchine, \textit{\"Uber eine Klasse linearer Diophantischer Approximationen}, Rend.\ Circ.\ Math.\ Palermo {\bf 50} (1926), 170--195.
   
   \bibitem{Kweights} D.\ Kleinbock, \textit{Flows on homogeneous spaces and Diophantine properties of matrices},
Duke Math.\ J.\ {\bf 95} (1998), no.\ 1, 107--124.

 
    \bibitem{K}
\bysame,
\textit{Quantitative nondivergence and its Diophantine applications}, 
in:   Homogeneous flows, moduli spaces and arithmetic, pages 131--153, 
Clay Math.\ Proc., 10, Amer. Math. Soc., Providence, RI, 2010. 

\bibitem{KLW} D. Kleinbock, E. Lindenstrauss, and B. Weiss, \textit{On
    fractal measures and diophantine approximation,} Selecta
  Math.\ (N.S.) {\bf 10} (2004), no.\ 4, 479--523. 
    
 \bibitem{KMa}
D.\  Kleinbock and G.\,A. Margulis,
\textit{Flows on homogeneous spaces and {D}iophantine approximation on
  manifolds},
Ann.\ of Math. (2), {\bf 148} (1998), no.\ 1, 339--360.

{ \bibitem{KM}
 D.\ Kleinbock and N.\ Moshchevitin,
 \textit{Simultaneous Diophantine approximation: sums of squares and homogeneous polynomials},
  Acta Arith.\ {\bf 190} (2019), 87--100.}
 
  \bibitem{KW}
D.\ Kleinbock and B.\ Weiss, \textit{Friendly  measures,  homogeneous  flows  and  singular vectors}, in:
Algebraic and Topological Dynamics, Contemp.\ Math.\
{\bf 211},  Amer.\
Math.\ Soc., Providence, RI, 2005, pp.\ 281--292.

   
 





 


 
 \bibitem{MaMo}
 A.\ Marnat and N.\ Moshchevitin,  \textit{An optimal bound for the ratio between ordinary and uniform exponents of Diophantine approximation}, preprint available  at {\tt arXiv:1802.03081v3} (2018).
 
 
   
   
   \bibitem{m}
   N.\ Moshchevitin, \textit{Khintchine's singular Diophantine systems and their applications}, Russian Math.\ Surveys {\bf 65} (2010), no.\ 3, 433--511.
   
   \bibitem{Ng}
   V.\ Nguyen, A.\  Poëls and D.\ Roy,  \textit{A transference principle for simultaneous rational approximation},
preprint available  at {\tt  arXiv:1908.11777}  (2019).
 
    \bibitem{R2u}
   A.\ Poëls,
   \textit{A 
   class of maximally singular sets for rational approximation}, preprint available at {\tt arXiv:1909.12159} (2019).
  
\bibitem{R1u}
A.\ Poëls and D.\ Roy \textit{Rational approximation to real points on quadratic hypersurfarces},   preprint available at {\tt arXiv:1909.01499} (2019).
   
   
   
  \bibitem{R1}
   D.\ Roy, \textit{Approximation simultan\'ee d'un nombre et de son carr\'e},  C.\ R.\ Math.\ Acad.\ Sci.\ Paris  {\bf 336} (2003), no.\ 1,  1--6.


  \bibitem{R2}
\bysame, \textit{ On Two Exponents of Approximation Related to a Real Number and Its Square},
Canad.\ J.\ Math.\ {\bf 59} (2007), no.\ 1, 211--224.
 
   
     \bibitem{SL}
   W.\,M.\ Schmidt,
   \textit{Badly approximable systems of linear forms}, J. Number Theory 1 (1969), 139-154.
   
   \bibitem{Sch}
   \bysame,
   \textit{Open problems in Diophantine Approximations}, in: Approximations Diophantiennes et nombres transcendants, Luminy, 1982,
   Progress in Mathematics, Birkh\"{a}user,  271--289 (1983).
   
     \bibitem{W} B.\ Weiss, \textit{Divergent trajectories on noncompact parameter spaces},  Geom.\ Funct.\ Anal.\ {\bf 14} (2004), no.\ 1, 94--149.

   
\end{thebibliography}
\end{document}